\documentclass[12pt]{amsart}
\usepackage{amsmath}
\usepackage{amssymb}
\usepackage{hyperref}
\usepackage{epsfig}
\usepackage{graphicx}
\numberwithin{equation}{section}
\usepackage{color}

\input xy
\xyoption{all}

\calclayout
\allowdisplaybreaks[3]

\theoremstyle{plain}
\newtheorem{prop}{Proposition}

\newtheorem{theo}[prop]{Theorem}

\theoremstyle{definition}

\newtheorem{rema}[prop]{Remark}

\newtheorem{exam}[prop]{Example}

\def\ra{\rightarrow}

\def\cI{{\mathcal I}}

\def\cN{{\mathcal N}}
\def\cO{{\mathcal O}}

\def\cX{{\mathcal X}}
\def\cY{{\mathcal Y}}

\def\mo{{\mathfrak o}}

\def\mp{{\mathfrak p}}
\def\mq{{\mathfrak q}}
\def\mr{{\mathfrak r}}
\def\fS{{\mathfrak S}}

\def\bP{{\mathbb P}}
\def\bQ{{\mathbb Q}}

\def\bZ{{\mathbb Z}}
\def\bR{{\mathbb R}}

\def\bC{{\mathbb C}}

\def\bF{{\mathbb F}}

\def\Br{\mathrm{Br}}
\def\CH{\mathrm{CH}}
\def\Ker{\mathrm{Ker}}

\def\CH{\mathrm{CH}}
\def\Pic{\mathrm{Pic}}

\def\bPGL{\mathrm{PGL}}

\def\Hom{\mathrm{Hom}}

\def\lim{\mathrm{lim}}

\def\Jac{\mathrm{Jac}}

\makeatother
\makeatletter

\author{Brendan Hassett}
\address{Department of Mathematics\\
Brown University \\
Box 1917 
151 Thayer Street
Providence, RI 02912 \\
USA}
\email{bhassett@math.brown.edu}

\author{Alena Pirutka}
\address{Courant Institute\\
                New York University \\
                New York, NY 10012 \\
                USA }
\email{pirutka@cims.nyu.edu}

\author{Yuri Tschinkel}
\address{Courant Institute\\
                New York University \\
                New York, NY 10012 \\
                USA }
\email{tschinkel@cims.nyu.edu}

\address{Simons Foundation\\
160 Fifth Avenue\\
New York, NY 10010\\
USA}

\title[Stable rationality]{Stable rationality of quadric surface bundles over surfaces}

\begin{document}
\date{\today}

\begin{abstract}
We study rationality properties of quadric surface bundles
over the projective plane. We exhibit families of smooth projective
complex fourfolds of this type over connected bases, containing both
rational and non-rational fibers.
\end{abstract}

\maketitle

\section{Introduction}

The main goal of this paper is to show that rationality
is not deformation invariant for families of smooth
complex projective varieties of dimension four.
Examples along these lines are known when singular
fibers are allowed, e.g., 
smooth cubic threefolds
(which are irrational) may specialize to cubic threefolds
with ordinary double points (which are rational), while
smooth cubic surfaces
(which are rational) may specialize
to cones over elliptic curves.
Totaro has constructed examples of this type with
mild singularities in higher dimensions \cite{totaro-15}.

Let $S$ be a smooth projective rational surface over the complex numbers 
with function field $K=\bC(S)$. A {\em quadric surface bundle} consists of a 
fourfold $X$ and a flat projective morphism $\pi:X\ra S$ such that the
generic fiber $Q/K$ of $\pi$ is a smooth quadric surface. We assume that $\pi$
factors through the projectivization of a rank four vector
bundle on $S$ such that the fibers are (possibly singular) quadric surfaces;
see Section~\ref{sect:quadric} for relevant background.

\begin{theo} \label{theo:main}
There exist smooth families of complex projective 
fourfolds $\phi: \mathcal X \ra B$ over 
connected varieties $B$, such that
for every $b\in B$ the fiber $\mathcal X_b=\phi^{-1}(b)$ 
is a quadric surface bundle over $\bP^2$, and satisfying:
\begin{enumerate}
\item for very general $b\in B$ the fiber $\mathcal X_b$ is not stably rational;
\item the set of points $b\in B$ such that $\mathcal X_b$ is rational is dense in $B$ for the Euclidean topology.  
\end{enumerate}
\end{theo}
Concretely, we consider smooth hypersurfaces 
$$X\subset \bP^2 \times \bP^3$$
of bidegree $(2,2)$; projection onto the first factor gives
the quadric surface bundle.

Our approach has two key elements. First, we apply the
technique of the decomposition of the diagonal
\cite{Voisin,ct-pirutka,ct-pir-cyclic,totaro-JAMS}
to show that very general $X\subset \bP^2 \times \bP^3$
of bidegree $(2,2)$ fail to be stably rational. 
The key is to identify a degenerate quadric surface
fibration, with non-trivial second unramified cohomology
and mild singularities. The analogous degenerate
conic bundles over $\bP^2$ are the Artin-Mumford
examples; deforming these allows one to show that
very general conic bundles of $\bP^2$ with large
degeneracy divisor fail to be stably rational \cite{HKT-conic}.
Second, quadric surface bundles are rational over the
base whenever
they admit a section, indeed, whenever they admit a 
multisection of odd degree. If the base is rational
then the total space is rational as well; this can
be achieved over a dense set of the moduli
space \cite{hassett-JAG,voisin-stable}. 
This technique also yields
rationality for a dense family of
cubic fourfolds containing a plane;
no cubic fourfolds have been shown not to be stably
rational.

This paper is inspired by the approach of Voisin
\cite{voisin-stable}, who also considers fourfolds
birational to quadric surface bundles. While our proof
of rationality is similar, the analysis of unramified
cohomology relies on work of Pirutka \cite{pirutka-survol}.

\

\noindent {\bf Acknowledgments:} The first author was 
partially
supported through NSF grant 1551514.

\section{Generalities}
\label{sect:gen}

We recall implications of 
the ``integral decomposition 
of the diagonal and specialization'' method, following 
\cite{ct-pirutka},  \cite{Voisin}, and \cite{pirutka-survol}. 

A projective variety $X$ over a field $k$ is {\em universally 
$\CH_0$-trivial} if for all field extensions $k'/k$ 
the natural degree homomorphism from the Chow group of zero-cycles
$$
\CH_0(X_{k'})\ra \bZ
$$
is an isomorphism. Examples include smooth $k$-rational varieties. More complicated examples arise as follows:

\begin{exam} \cite[Lemma 2.3, Lemma 2.4]{ct-pir-cyclic}
\label{exam:trivial}
Let $X=\cup_i X_i$ be a projective, reduced, geometrically connected variety over a field $k$ such that:
\begin{itemize}
\item Each component $X_i$ is geometrically irreducible and $k$-rational, with isolated singularities.
\item Each intersection $X_i\cap X_j$ is either empty or has a zero-cycle of degree 1. 
\end{itemize}
Then $X$ is universally $\CH_0$-trivial. 
\end{exam}

A projective morphism 
$$
\beta: \tilde{X}\ra X
$$ 
of $k$-varieties 
is {\em universally $\CH_0$-trivial} if for all extensions
$k'/k$ the push-forward homomorphism
$$
\beta_* : \CH_0(\tilde{X}_{k'})\ra \CH_0(X_{k'})
$$
is an isomorphism.

\begin{prop}\cite[Proposition 1.8]{ct-pirutka} \label{prop:restrivial}
Let 
$$
\beta:\tilde{X}\ra X
$$ 
be a projective morphism such that for every scheme point $x$ of $X$, 
the fiber $\beta^{-1}(x)$, considered as a variety over the residue field $\kappa(x)$,  
is universally $\CH_0$-trivial. Then $\beta$ is universally $\CH_0$-trivial. 
\end{prop}

For example, if $X$ is a smooth projective variety and 
$$
\beta: \mathrm{Bl}_Z(X)\ra X
$$ 
is a blowup of a smooth subvariety $Z\subset X$,
then $\beta$ is 
a universally $\CH_0$-trivial morphism,
since all fibers over (scheme) points are projective spaces.
More interesting examples arise as resolutions
of singularities of certain singular projective varieties.

Examples of failure of universal $\CH_0$-triviality are given by smooth projective varieties $X$ with nontrivial Brauer group $\Br(X)$, or more generally, 
by varieties with nontrivial higher unramified cohomology \cite[Section 1]{ct-pirutka}. The following specialization argument is the key to recent advances in 
investigations of stable rationality: 

\begin{theo}
\label{thm:main-help}
\cite[Theorem 2.1]{Voisin}, \cite[Theorem 2.3]{ct-pirutka}
Let 
$$
\phi: \mathcal X\ra B
$$ 
be a flat projective morphism  of complex varieties with smooth generic fiber. 
Assume that there exists a point $b\in B$ such that the fiber 
$$
X:=\phi^{-1}(b)
$$ 
satisfies the following conditions:
\begin{itemize}
\item    the group   $H^2_{nr}(\mathbb C(X), \bZ/2)$ is nontrivial;
\item $X$ admits a desingularization 
$$
\beta: \tilde{X}\ra X
$$
such that the morphism $\beta$ is universally $\CH_0$-trivial. 
\end{itemize}
Then a very general fiber of $\phi$ is not stably rational.
\end{theo}

\section{Quadric surface bundles}
\label{sect:quadric}

Let $S$ be a smooth projective variety over $\bC$. 
Suppose that $\pi:X\ra S$ is a quadric surface bundle, i.e.,
a flat projective morphism from a variety such that
the generic fiber $Q$ is a smooth quadric surface.
We assume it admits a factorization
$$X \hookrightarrow \bP(V) \rightarrow S,$$
where $V\ra S$ is a rank four vector bundle and the fibers
of $\pi$ are expressed as quadric surfaces in the fibers of $\bP(V) \ra S$.
There is a well-defined degeneracy divisor $D\subset S$ corresponding
to where the associated quadratic form drops rank.  

Trivializing $V$ over an open neighborhood of $S$, $X$ may be expressed using
a symmetric $4\times 4$ matrix $(a_{ij})$:
$$
\sum a_{ij}x_ix_j=0.
$$
The local equation for $D$ is
the determinant $\det((a_{ij})).$ 
Note that $D$ has multiplicity $\ge 2$ where the rank of fibers is less
than three. Indeed, the hypersurface
$$\{\det(a_{ij})=0 \} \subset \bP^9_{(a_{ij})}$$
is singular precisely where all the $3\times 3$ minors vanish.

\

It is well known that $Q$ is rational over $K=\bC(S)$ if and only if $Q(K)\neq \emptyset$,
i.e., when $\pi$ admits a rational section. A theorem of Springer \cite{Springer} implies
that $Q(K)\neq \emptyset$ provided $Q(K')\neq \emptyset$ for some extension $K'/K$ of
odd degree, i.e., when $\pi$ admits a rational multisection of odd degree.
Thus we obtain 

\begin{prop} \label{prop:rattest}
Let $\pi:X \ra S$ be a quadric surface bundle as above, 
with
$S$ rational. Then $X$ is rational provided $\pi$
admits a multisection of odd degree.
\end{prop}

Let $F_1(X/S) \ra S$ denote the relative variety of lines
of $\pi$. Let $S_{\circ} \subset S$ denote the largest
open subset such that $S_{\circ} \cap D$ is smooth
and $X_{\circ}=X\times_S S_{\circ}.$
Then 
$F_1(X_{\circ}/S_{\circ}) \ra S_{\circ}$
factors 
$$F_1(X_{\circ}/S_{\circ}) \stackrel{p}{\ra} T_{\circ} \ra S_{\circ},$$
where the second morphism is a double cover branched
along $S_{\circ}\cap D$ and the first morphism
is an \'etale $\bP^1$-bundle. 
In particular $F_1(X_{\circ}/S_{\circ})$ is non-singular.
Let $\alpha \in \Br(T_{\circ})[2]$ denote the
Brauer class arising from $p$.

Let $F$ be a 
resolution of the closure of $F_1(X_{\circ}/S_{\circ})$
in $F_1(X/S)$ obtained by blowing up over the complement
of $S_{\circ}$. 
The incidence correspondence between $X$ and $F_1(X/S)$
$$\Gamma' \subset X \times_S F_1(X/S)$$
induces a correspondence $\Gamma$ between $X$ and $F$
and a homomorphism
$$\Gamma_*:\CH^2(X) \ra \Pic(F).$$
Let $\eta$ denote the generic point of $S$; there is a homomorphism
$$\Xi_*:\Pic(F_{\eta}) \ra \CH^2(X_{\eta}).$$
constructed is follows: Consider $Z\subset F_{\eta}$ a finite reduced
subscheme with support on each component of $F_{\eta}$, e.g., a choice
of $n$ lines from each ruling.
Take the union of the corresponding rulings in $X_{\eta}$ and 
set $\Xi_*(Z) \subset X_{\eta}$ to be its singular locus, e.g., $n^2$
points where the rulings cross. This is compatible with rational 
equivalence and yields the desired homomorphism.
Thus a divisor with odd degree on each geometric component of 
$F_{\eta}$ gives rise to a multisection of odd degree.

The correspondences $\Gamma$ and $\Xi$ 
guarantee the following conditions
are equivalent:
\begin{itemize}
\item{$\alpha=0;$}
\item{$F$ admits a divisor intersecting the generic fiber $F_{\eta}$
with odd degree on each component;}
\item{$X$ admits a multisection of odd degree.}
\end{itemize}
The correspondence $\Gamma$ also
acts at the level of Hodge classes; here we obtain an
equivalence:
\begin{itemize}
\item{$F$ admits an integral $(1,1)$-class
intersecting the generic fiber $F_{\eta}$
with odd degree on each component;}
\item{$X$ admits an integral $(2,2)$-class
intersecting the fibers of $\pi$ with odd degree.}
\end{itemize}
Applying the Lefschetz $(1,1)$ Theorem to $F$
and Proposition~\ref{prop:rattest}
we obtain: 
\begin{prop} \label{prop:ratmain}
Let $\pi:X \ra S$ be a quadric surface bundle as above, with $X$ smooth
and $S$ rational.
Then $X$ is rational if it
admits an integral $(2,2)$-class meeting the fibers
of $\pi$ in odd degree.
\end{prop}
\begin{rema}
See \cite[Cor.~8.2]{CTV} for results on the integral Hodge
conjecture for quadric bundles over surfaces; these 
suffice for our application to quadric surface bundles over $\bP^2$.
\end{rema}

The generic fiber of $\pi$ is a quadric surface, that admits a diagonal form
\begin{equation}
\label{eqn:qq}
Q=<1,a,b,abd>, 
\end{equation}
i.e., is given by the equation
$$                                                                                                        
s^2+at^2+bu^2+abd v^2=0                                                                            
$$
where $a,b,d\in K^\times$ and $(s, t,u,v)$ are homogeneous coordinates in $\bP^3$.  Note that since $k:=\mathbb C\subset K$, this form is equivalent to the form  $<1,-a,-b,abd>$.

Theorem 3.17 in \cite{pirutka-survol} gives a general formula for the unramified $H^2$ of the field $K(Q)$, in terms
of the divisor of rational functions $a,b,d\in K^\times$, under the assumption that $d$ is not a square. 

In Section~\ref{sect:brauer} we will analyze the following special case:

\begin{exam}
\label{exam:basic}
Consider the fourfold $X\subset \bP^2\times \bP^3$ given by
\begin{equation}
\label{eqn:basic}
 yz s^2 + xz t^2 + xy u^2 + F(x,y,z) v^2 = 0,
\end{equation}
where 
$$
F(x,y,z)=x^2+y^2+z^2 - 2(xy+xz+yz).
$$
Dehomogenize by setting $z=1$ to obtain a quadric surface over $k(\bP^2)$:
$$y s^2 + x t^2 + xy u^2 + F(x,y,1) v^2 =0.$$
Multiplying through by $xy$ and absorbing squares into the variables
yields
$$
x S^2 + y T^2 + U^2 + xy F(x,y,1) V^2 =0,
$$
which is of the form (\ref{eqn:qq}).

We compute the divisor $D\subset \bP^2$
parametrizing singular fibers of $\pi: X\ra \bP^2$. This is
is reducible, consisting of the coordinate lines (with multiplicty two)
and a conic tangent to each of the lines:
$$D=\{ x^2y^2z^2(x^2+y^2+z^2-2(xy+xz+yz))=0 \}.$$
\end{exam}

We will therefore be interested in quadric surface
bundles over $\bP^2$ degenerating over octic plane
curves $D'\subset \bP^2$. 
The corresponding double covers $T\ra \bP^2$ are {\em Horikawa} surfaces, i.e.,  
minimal surfaces of general type such that $c_1^2=2\chi-6$. 
Their moduli space is irreducible. The paper \cite{persson} gives
examples of such surfaces with maximal Picard numbers, obtained by degenerating 
$D'$.

\begin{prop}
\label{prop:deforms}
For a generic octic curve $D'\subset \bP^2$ 
there exists a smooth curve $B$ and a flat family   
$$
\phi:\mathcal X\ra B
$$
of hypersurfaces in $\bP^2 \times \bP^3$ of bidegree
$(2,2)$ and points $b,b'\in B$ 
such that 
\begin{itemize}
\item $X=\phi^{-1}(b)$ is the quadric bundle from Example~\ref{exam:basic} and
\item $\pi':X'=\phi^{-1}(b')\ra \bP^2$ is smooth and degenerates exactly over $D'$. 
\end{itemize}
\end{prop}

\begin{proof} Essentially, we can smooth hypersurfaces
in products of projective spaces freely.

In detail, the parameter space of $(2,2)$ hypersurfaces 
in $\bP^2 \times \bP^3$ is a projective space of
dimension $59$; the moduli space $\cN_{(2,2)}$,
obtained by taking a
quotient for the natural action of $\bPGL_3 \times \bPGL_4$,
has dimension $36$. The parameter space of octic
plane curves is a projective space of dimension $44$;
the moduli space $\cN_8$ has dimension $36$. 

Taking determinants gives a rational map
$$\bP^{59} \dashrightarrow \bP^{44}$$
which descends to 
$$\varphi: \cN_{(2,2)} \dashrightarrow \cN_8.$$
This is known to be 
dominant \cite[Proposition 4.6]{beau-det}; indeed, if $D'$ is generic
then the fiber $\varphi^{-1}([D'])$ 
corresponds to the 
non-trivial two-torsion of the Jacobian of $D'$.
Furthermore, each $X' \in \phi^{-1}([D'])$ is smooth.  

Clearly, there is a curve $\ell \subset \bP^{44}$ joining
$[D']$ to $[D]$. Note that $\varphi([X])=[D]$; the
morphism $\varphi$ is clearly defined at $[X]$. 
Let $B$ be the normalization of an irreducible component
of $\cN_{(2,2)}\times_{\cN_8}\ell$ containing $[X]$.
Its image in $\cN_8$ contains $[D']$ by construction.
\end{proof}

\begin{rema} \label{rema:conic}
Hypersurfaces of bidegree $(2,2)$ in $\bP^2 \times \bP^3$ may also
be regarded as conic bundles over the second factor. The degeneracy 
surface in $\bP^3$ has degree six and at least eight nodes, corresponding
to rank-one fibers. As a byproduct of the proof of Theorem~\ref{theo:main},
we obtain failure of stable rationality for very general conic bundles of this
type.
\end{rema}

\section{The Brauer group of the special fiber}
\label{sect:brauer}

Let $K$ be a field. We write 
$$
H^n(K)=H^n(K, \bZ/2)
$$ 
for its $n$-th Galois cohomology with constant coefficients $\bZ/2$.
Let $K=k(X)$ be the function field of an algebraic variety $X$ over $k=\bC$, and let
$\nu$ be a divisorial (rank-one) valuation of $K$. For $n\ge 1$, 
we have a natural homomorphism
$$
\partial^n_{\nu}:  H^n(K)\ra H^{n-1}(\kappa(\nu)), 
$$
where  $\kappa(\nu)$ is the residue field of $\nu$. 
The group
$$
H^n_{nr}(K):=\cap_{\nu} \,\,\Ker(\partial^n_{\nu})
$$
is called the $n$-th unramified cohomology of $K$. It is a  stable birational invariant, by definition; 
it vanishes if $X$ is   stably rational.
Recall that for smooth projective $X$ we have
$$
\Br(X)[2]=H^2_{nr}(k(X)). 
$$

The following proposition is similar to the examples in  \cite[Section 3.5]{pirutka-survol}.

\begin{prop}
\label{prop:brauer}
Let $K=k(x,y)=k(\mathbb P^2)$, 
$X\ra \bP^2$ the quadric surface bundle defined in 
Example~\ref{exam:basic}, 
$$
\alpha= (x,y)\in \Br(K)[2],
$$  
and $\alpha'$ its image in $H^2(k(X))$.
Then $\alpha'$ is contained in 
$H^2_{nr}(k(X))$ and is nontrivial; in particular,
$$
H^2_{nr}(k(X))\neq 0.
$$ 
\end{prop}

\begin{proof}
Let $Q$ be the generic fiber of the natural projection $\pi: X\to \mathbb P^2$. 
Since the discriminant of $Q$ is not a square, the homomorphism 
$$
H^2(K)\to H^2(K(Q))
$$ 
is injective \cite{arason}.
Thus we have to show that for every divisorial (rank-one) valuation 
$\nu$ on $K(Q)$ we have $\partial_\nu(\alpha)=0$. (For simplicity, we write
$\partial_{\nu}$ for $\partial^2_{\nu}$.)
We use  standard coordinates 
$x$ and $y$ (resp.~$y$ and $z$, resp.~$x$ and $z$) for the open charts of the projective plane.
Let us first investigate the ramification of 
$\alpha$ on $\mathbb P^2$; 
from the definition, we only have the following nontrivial residues:
\begin{itemize}
\item  $\partial_x(\alpha)=y$ at the line $L_x: x=0$, 
where we write $y$ for its class in the residue field $k(y)$ modulo squares;
\item $\partial_y(\alpha)=x$ at the line $L_y: y=0$,
\item $\partial_z(\alpha)=\partial_z(z, zy)=y$ at the line $L_z: z=0$, 
in coordinates $y$ and $z$ on $\bP^2$. 
\end{itemize}
Let $\mo_{\nu}$ be the valuation ring of $\nu$ in $K(Q)$ and consider
the center of $\nu$ in  $\mathbb P^2$. 
If $\mo_{\nu} \supset K$ then the $\partial_{\nu}(\alpha)=0$; hence there
are two cases to consider: 
\begin{itemize}
\item 
The center is the generic point of a curve $C_\nu$; we denote the corresponding
residue map $\partial_{C_{\nu}}:H^2(K) \ra H^{1}(\kappa(C_\nu)).$
\item 
The center is a closed point $\mp_\nu$.
\end{itemize}

\

{\it Codimension $1$}. 
The inclusion of discrete valuation rings 
$\mo_{\mathbb P^2, C_\nu}\subset \mo_{\nu}$ induces a commutative diagram
\begin{equation}\label{diagr}
\xymatrix{
H^2(K(Q))\ar[r]^{\partial_{\nu}} &H^{1}(\kappa(\nu))&\\
H^2(K)\ar[r]^{\partial_{C_{\nu}}}\ar[u] &H^{1}(\kappa(C_\nu))\ar[u] &
}
\end{equation}
Hence we have the following cases:
\begin{enumerate}
\item $C_\nu$ is different from  $L_x, L_y$, or $L_z$. 
Then $\partial_{C_\nu}(\alpha)=0$, so that $\partial_\nu(\alpha')$ is zero from the diagram above.
\item $C_\nu$ is one of the lines  $L_x, L_y$ or $L_z$. Note that modulo the equation of $C_\nu$, 
the element  $d:=F(x,y,z)$  is a nonzero square, 
so that \cite[Cor.~3.12]{pirutka-survol} gives $\partial_\nu(\alpha')=0$.
\end{enumerate}
We deduce that for any valuation $\nu$ of $K(Q)$ 
with center a codimension $1$ point in $\mathbb P^2_{\mathbb C}$ 
the residue $\partial_\nu(\alpha')$ vanishes.

\

{\it Codimension $2$.} Let $\mp_\nu$ be the center of $\nu$ on $\mathbb P^{2}$.
We have an inclusion of local rings  $\mo_{\mathbb P^2, \mp_\nu}\subset \mo_{\nu}$ 
inducing the inclusion of corresponding completions 
$\widehat{\mathcal O_{\mathbb P^2, \mp_\nu}}\subset  \widehat{\mo_{\nu}}$ 
with quotient fields $K_{\mp_\nu}\subset K(Q)_\nu$ respectively. 
We have three possibilities:
\begin{enumerate}
\item If $\mp_\nu\notin L_x\cup L_y\cup L_z$, then $\alpha$ is a cup product of units in 
$\mathcal O_{\mathbb P^2, \mp_\nu}$, hence units in $\mo_v$, so that $\partial_{\nu}(\alpha')=0$.
\item If $\mp_{\nu}$ lies on one curve, e.g., 
$\mp_\nu\in L_x\setminus (\mp_y \sqcup \mp_z)$,  where  $\mp_y=(0,1,0)$ and $\mp_z=(0,0,1)$,  
then the image of $y$ in $\kappa(\mp_\nu)$ is a nonzero complex number, hence a square in   
$\widehat{\mathcal O_{\mathbb P^2, \mp_\nu}}$, and $y$ is also a square in $\widehat{\mo_\nu}$. 
Thus $\alpha'=0$ in $H^2(K(Q)_\nu, \mathbb Z/2)$, and $\partial_\nu(\alpha')=0$.
\item If $\mp_\nu$ lies on two curves, e.g., $\mp_\nu=L_x\cap L_y$, 
then the image of $F(x,y,1)$ in $\kappa(\mp_\nu)$ is a nonzero complex number, hence a square.
By \cite[Corollary 3.12]{pirutka-survol}, we have $\partial_\nu(\alpha')=0$.
\end{enumerate}

\end{proof}

\section{Singularities of the special fiber}
\label{sect:sing}

In this section we analyze the singularities of the fourfold introduced in Section~\ref{sect:brauer}. 
Our main result is:

\begin{prop}
\label{prop:resolve}
The fourfold $X\subset \bP^2\times \bP^3$, with coordinates $(x,y,z)$ and $(s,t,u,v)$, respectively, 
given by 
\begin{equation}
\label{eqn:special}
yzs^2+xzt^2+xyu^2+F(x,y,z) v^2 = 0, 
\end{equation}
with 
\begin{equation}
\label{eqn:F}
F(x,y,z)= x^2+y^2+z^2 -2 (xy+xz+yz),
\end{equation}
admits a $\CH_0$-trivial resolution of singularities. 
\end{prop}

We proceed as follows:
\begin{itemize}
\item identify the singular locus of $X$;
\item construct a resolution of singularities $\beta: \tilde{X}\to X$;
\item verify universal $\CH_0$-triviality of $\beta$. 
\end{itemize}

\subsection{The singular locus}
\label{sect:sing-loc}

Here we describe the singularities explicitly using affine
charts on $\bP^2 \times \bP^3$. The equations
(\ref{eqn:special}) and (\ref{eqn:F}) are symmetric
with respect to compatible permutations of $\{x,y,z\}$
and $\{s,t,u\}$. In addition, there is the symmetry
$$(s,t,u,v) \ra (\pm s, \pm t, \pm u, v)$$
so altogether we have an action by the semidirect
product $(\bZ/2\bZ)^3 \rtimes \fS_3$.

\subsubsection*{Analysis in local charts}
Let $L_x, L_y, L_z\subset \bP^2$ be the coordinate lines given by
$$
x=0, \quad y=0, \quad z=0, 
$$
respectively, and 
$$
\mp_x:=(1,0,0),\quad \mp_y:=(0,1,0), \quad \mp_z:=(0,0,1)
$$ 
their intersections.

The quadrics in the family \eqref{eqn:special} drop rank over coordinate lines $L_x, L_y, L_z$ 
and over the conic $C\subset \bP^2$, with equation \eqref{eqn:F} $$F(x,y,z)=0.$$
This conic is tangent to the coordinate lines in the points 
$$
\mr_x:=(0,1,1),\quad \mr_y:=(1,0,1), \quad \mr_z:=(1,1,0),
$$ 
respectively.  

\

By symmetry, if suffices to consider just two
affine charts:

\begin{itemize}
\item [{\bf Chart 1:}]
$z=u=1$. 
Equation~\eqref{eqn:special} takes the form 
\begin{equation}
\label{eqn:seq}
ys^2+xt^2+xy+F(x,y,1)v^2=0.
\end{equation}
Derivatives with respect to $s,t,v$ give
 \begin{equation}
 \label{pd}
 ys=0, \quad xt=0, \quad   vF(x,y,z)=0.
 \end{equation}
 Hence $xy=0$, from \eqref{eqn:seq}.
 Derivatives with respect to $y, x$ give
 \begin{equation}
 \label{pd1}
 s^2+x+(2y-2x-2)v^2=0,\quad t^2+y+(2x-2y-2)v^2=0.
 \end{equation}
 Since $xy=0$, we  have two cases, modulo symmetries:
 \begin{itemize}
 \item [{\em Case 1:}] $y=0$;
 \item [{\em Case 2:}] $x=0, y\neq 0$.
 \end{itemize}
 We analyze each of these cases:
  \begin{itemize}
 \item [{\em Case 1:}] $y=0$. 
 Then  $vF(x,y,z)=0$ (from (\ref{pd})) implies  
$$v(x-1)^2=0.$$
 So either $v=0$ or $x=1$. 
 If $v=0$,  from (\ref{pd1}) we obtain  $s^2+x=t=0$. 
 Hence we obtain the following equations for the
   singular locus:
  \begin{equation}\label{sing0}
 y=v=t=s^2+x=0.
 \end{equation}
 
 If $x=1$ then (\ref{pd}) implies $t=0$, 
 and  the remaining equation from (\ref{pd1}) gives $s^2+1-4v^2=0$.  Hence we obtain the following equations:
 \begin{equation}\label{sing1}
 x-1=y=t=s^2+1-4v^2=0.
 \end{equation}
 \item [{\em Case 2:}] $x=0, y\neq 0$. 
 From  (\ref{pd}) the condition $ys=0$ implies $s=0$. 
 There are two more cases: 
 $v=0$ or $v\neq 0$.  
 If $v=0$ the remaining equation (\ref{pd1}) gives $t^2+y=0$. Hence we obtain equations for the singularity: 
 \begin{equation}\label{sing2}
 x=v=s=t^2+y=0.
 \end{equation}
 
 If $v\neq 0$, then $F(0, y, 1)=(y-1)^2=0$ from (\ref{pd}), hence $y=1$. 
 The remaining equation from   (\ref{pd1}) gives 
 $$
 t^2+y+(2x-2y-2)v^2=t^2+1-4v^2=0.
 $$ 
 So we obtain equations for the singular locus:
  \begin{equation}\label{sing3}
 x=  y-1=s=t^2+1-4v^2=0.
 \end{equation} 
\end{itemize}

\item [{\bf Chart 2:}] $z=s=1$.
The equation of the quadric bundle is
$$
y+xt^2+xyu^2+F(x,y,1)v^2=0.
$$
As above, derivatives with respect to $t, v, u$ give
\begin{equation}
\label{pd2}
 xt=0, \quad xyu=0, \quad vF=0.
\end{equation}
Thus $y=0$ from the equation.
The conditions above and derivatives with respect to $x$ and $y$ yield
\begin{equation}
\label{pd3}
xt=v(x-1)^2=1+xu^2+(-2x-2)v^2=t^2+(2x-2)v^2=0.
 \end{equation}
The second equation implies that either $x=1$ or $v=0$.
 
 \
 
 \noindent
 If $x=1$, we obtain: 
 \begin{equation}\label{sing4}
 x-1=y=t=1+u^2-4v^2=0.
 \end{equation} 
 If $v=0$, we obtain:
  \begin{equation}\label{sing5}
 y=t=v=1+xu^2=0.
 \end{equation} 
\end{itemize}
 
 \

Collecting these computations, we obtain the following singularities:
 \begin{enumerate}
 \item 
 In the chart $ys^2+xt^2+xy+F(x,y,1)v^2=0:$
\begin{align*}
  C_{y,s}^\circ: & \    y=v=t=s^2+x=0 \\
  R_{y,s}^\circ: & \  x-1=y=t=s^2+1-4v^2=0 \\
  C_{x,t}^\circ: & \ x=v=s=t^2+y=0 \\
 R_{x,t}^\circ: & \ x=y-1=s=t^2+1-4v^2=0
\end{align*}
 \item In the chart $y+xt^2+xyu^2+F(x,y,1)v^2=0:$  
\begin{align*}
  R_{y,u}^\circ: &\   x-1=y=t=1+u^2-4v^2=0 \\
C_{y,u}^\circ: & \  y=t=v=1+xu^2=0
\end{align*}
\end{enumerate}

\subsubsection*{Enumeration of strata}
The singular locus of $X$ is a union of 6 conics. 
We distinguish between
\begin{itemize}
\item 
Horizontal conics $C_x,C_y, C_z\subset X$: 
these project onto the coordinate lines $L_x,L_y,L_z\subset \bP^2$. We express them using our standard coordinates
on $\bP^2 \times \bP^3$:
\begin{align*}
C_y=&\{  y=t=v=0, \ zs^2+xu^2=0  \}  \\
C_x=&\{  x=s=v=0, \ zt^2+yu^2=0  \}  \\
C_z=&\{  z=u=v=0, \ xt^2+ys^2=0  \}  
\end{align*}
The conics intersect transversally 
over $\mp_z$, $\mp_x, \mp_y\in \bP^2$, respectively:
\begin{align*}
C_x\cap C_y=&\mq_z:=(0,0,1)\times (0,0,1,0), \quad \pi(\mq_z)=\mp_z \\
C_y\cap C_z=&\mq_x:=(1,0,0)\times (1,0,0,0), \quad \pi(\mq_x)=\mp_x \\
C_x\cap C_z=&\mq_y:=(0,1,0)\times (0,1,0,0), \quad \pi(\mq_y)=\mp_y 
\end{align*}
 \item 
Vertical conics $R_y,R_x, R_z\subset X$: these project to the points 
$\mr_y, \mr_x, \mr_z\in \bP^2$:
\begin{align*}
R_y=&\{ x-z=y=t=0, \ s^2+u^2-4v^2=0\} \\
R_x=&\{ y-z=x=s=0, \ t^2+u^2-4v^2=0\} \\
R_z=&\{ x-y=z=u=0, \ s^2+t^2-4v^2=0\} 
\end{align*}
Vertical conics intersect
the corresponding horizontal conics 
transversally in two points:
\begin{align*}
R_y\cap C_y =& \{ \mr_{y+}, \mr_{y-}\} = 
			(1,0,1) \times (\pm i,0, 1, 0) \\
R_x\cap C_x =& \{ \mr_{x+}, \mr_{x-}\} =
			(0,1,1) \times (0,1,\pm i ,0) \\
R_z\cap C_z =& \{ \mr_{z+}, \mr_{z-}\} =
			(1,1,0) \times (1,\pm i, 0 ,0) 
\end{align*}
\end{itemize}

\subsubsection*{Local \'etale description of the singularities}
The structural properties of the resolution become clearer after
identifying \'etale normal forms for the singularities.

We now provide a local-\'etale description of the neighborhood of $\mq_z$. 
Equation \eqref{eqn:seq} takes the form
$$
ys^2+xt^2+xy+F(x,y,1)v^2=0.
$$
At $\mq_z$ we have $F(x,y,1)\neq 0$, so we can set
$$
v_1=\sqrt{F(x,y,1)}v 
$$
to obtain
$$
ys^2+xt^2+xy+v_1^2=0.
$$
Set $x=m-n$ and $y=m+n$ to get
$$
(m+n)s^2 + (m-n)t^2 + m^2 - n^2 + v_1^2 = 0
$$
or
$$
m(s^2+t^2) + n(s^2-t^2) + m^2 - n^2 + v_1^2 = 0.
$$
Then let 
$$
m=m_1-(s^2+t^2)/2 \quad \text{ and } \quad n=n_1+(s^2-t^2)/2
$$ 
to obtain
\begin{equation}
\label{eqn:NF1}
m_1^2 - n_1^2 + v_1^2 = ((s^2+t^2)^2 - (s^2-t^2)^2)/4 = s^2 t^2.
\end{equation}

We do a similar analysis in an \'etale-local neighborhood
of either of the points $\mr_{y\pm}$. The singular
strata  for $C_y$ and $R_y$ are given in (\ref{sing0}) and (\ref{sing1}):
$$\{y=t=v=s^2+x=0 \}, \{y=t=x-1=s^2+1-4v^2=0\}.$$
We first introduce a new coordinate $w=x-1$.
Thus the singular stratum is the intersection
of the monomial equations $y=t=vw=0$ and the hypersurface
$$s^2+w+1-4v^2.$$
We regard this as a local coordinate near $\mr_{y\pm}$.
Equation (\ref{eqn:seq}) transforms to
$$ys^2+wt^2+t^2+wy+y+v^2(-4y+(w-y)^2)=0.$$
Regroup terms to obtain
$$y(s^2+w+1-4v^2)+t^2(1+w)=-v^2(w-y)^2.$$
Let $t_1=t\sqrt{1+w}$, $s_1=s^2+w+1-4v^2$, and 
$w_1=w-y$ we obtain
\begin{equation}
\label{eqn:NF2}
ys_1+t_1^2=-v^2w_1^2.
\end{equation}

The normal forms (\ref{eqn:NF1}) and (\ref{eqn:NF2}) are
both equivalent to 
$$a_1^2+a_2^2+a_3^2=(b_1b_2)^2,$$
with ordinary threefold double points along the lines
$$\ell_1=\{a_1=a_2=a_3=b_1=0 \}, \quad
\ell_2 = \{a_1=a_2=a_3=b_2=0 \}.$$
A direct computation shows this is resolved by
blowing up $\ell_1$ and $\ell_2$ in {\em either} order.
The exceptional fibers over the generic points of
$\ell_1$ and $\ell_2$ are smooth quadric surfaces,
isomorphic to the Hirzebruch surface $\bF_0$. 
Over the origin, we obtain a copy
$$\bF_0 \cup_{\Sigma} \bF_2$$
where $\Sigma\simeq \bP^1$ is the $(-2)$-curve on
$\bF_2$ and has self-intersection $2$ on $\bF_0$.

By symmetry,
this analysis is valid at {\em all} nine special points
$$\mq_x,\mq_y,\mq_z,\mr_{x\pm}, \mr_{y\pm}, \mr_{z\pm}$$
where components of the singular locus 
(the horizontal and vertical conics) intersect. This
explains why we can blow these conics up in any
order.

\subsection{Resolution of singularities}
\label{sect:resolve}

\subsubsection*{What we need to compute}
We propose blowing up as follows:
\begin{enumerate}
\item{blow up $C_y$;}
\item{blow up the proper transform of $C_x$;}
\item{blow up the proper transform of $C_z$;}
\item{blow up the union of $R_x,R_y,$ and $R_z$,
which are disjoint.}
\end{enumerate}
Taking into account the symmetry, after the first step 
we must understand:
\begin{itemize}
\item{What are the singularities along the proper transform
of $C_x$?}
\item{What are the singularities along the proper transform
of $R_y$?}
\end{itemize}
Of course, answering the first questions clarifies the 
behavior along the proper transform of  $C_z$. And $R_x$
and $R_z$ behave the exactly the same as $R_y$.

Let $X_1$ denote the blow up of $C_y$ and $E_{1,y}$
the resulting exceptional divisor. We shall see that
\begin{itemize}
\item
$X_1$ is smooth at any point of the exceptional divisor 
$E_{1,y}$, except where $E_{1,y}$ meets the proper 
transforms of $C_x, C_z, R_y$.
\item
$E_{1,y}$ is 
also smooth, except where it meets the proper transforms
of $C_x, C_z, R_y$.
\item
The fibers of $E_{1,y}\to C_y$ are smooth
quadric surfaces away from $\mq_x,\mq_z,\mr_{y\pm}$,
over which the fibers are quadric cones.
\end{itemize}
Incidentally, $E_{1,y} \ra C_y$ admits sections and
$E_{1,y}$ is rational.

\subsubsection*{First blow up---local charts}

We describe the blow up of $C_y$ in charts.
We start in Chart 1, where $z=u=1$. Local equations for the
center are given in (\ref{sing0}) and we have a local chart
for each defining equation.

\begin{itemize}

\item {\bf Chart associated with $y$:}
Equations for the blow up of the ambient
space take the form
$$v=yv_1, \ t=yt_1, \ s^2+x=yw_1.$$ 
The equation of the proper transform of the quadric bundle is
$$w_1+xt_1^2+F(x,y,1)v_1^2=0,  s^2+x=yw_1.$$
The exceptional divisor $E_{1,y}$ is given by $y=0$, i.e., 
 $$
 w_1+xt_1^2+(x-1)^2v_1^2=0, \ s^2+x=0.
 $$
The blow up is smooth at any point of the exceptional divisor
in this chart. 
(The proper transforms of $R_y$ and $C_x$
do not appear in this chart.)
We analyze $E_{1,y}\ra C_y$:
 for any field $\kappa/\mathbb C$ and $a\in \kappa$,
the fiber above $s=a$, $x=-a^2$,  $y=v=t=0$, is given by
\begin{equation}\label{f1}
w_1-a^2t_1^2+(1+a^2)^2v_1^2=0,
\end{equation}
which is smooth in this chart.
Equation (\ref{f1}) makes clear 
that the exceptional divisor is rational and
admits a section over the center.  

\item {\bf Chart associated with $s^2+x$:} 
Equations for the blow up of the ambient
space take the form
$$y=(s^2+x)y_1,  v=(s^2+x)v_1, t=(s^2+x) t_1.$$
The proper transform of the quadric bundle has equation
$$y_1+xt_1^2+F(x,(s^2+x)y_1,1)v_1^2=0.$$
The exceptional divisor $E_{1,y}$ satisfies
$$y_1+xt_1^2+(x-1)^2v_1^2=0, \ s^2+x=0.$$
The blow up 
is smooth at any point of the exceptional divisor
in this chart, as the derivative with respect to $y_1$ is $1$
and the derivative of the second equation with respect to $y_1$ ( resp. $x$) is $0$ (resp. $1$).
(Again, the proper transforms of $R_y$ and $C_x$
do not appear in this chart.)
The fiber above $s=a$, $x=-a^2$,  $y=v=t=0$, is given by
\begin{equation}\label{f2}
y_1-a^2t_1^2+(1+a^2)^2v_1^2=0,
\end{equation}
which is smooth and rational in this chart.

\item {\bf Chart associated with $t$:}
Equations for the blow up of the ambient space are
$$y=ty_1, \ v=tv_1, \  s^2+x=tw_1$$
and the proper transform of the quadric bundle satisfies
$$y_1w_1+x+F(x,ty_1,1)v_1^2=0,  s^2+x=tw_1.$$
The exceptional divisor is given by $t=0$, i.e. 
$$y_1w_1+x+(x-1)^2v_1^2=0, \  s^2+x=0.$$
The blow up 
is smooth along the exceptional divisor, 
except at the point
$$t=v_1=y_1=s=w_1=x=0,$$
which lies over the point $\mq_z$. 
Thus the only singularity is along the proper transform
of $C_x$. 
The fiber above $s=a$, $x=-a^2$,  $y=v=t=0$, is given by
\begin{equation}\label{f3}
y_1w_1-a^2+(1+a^2)^2v^2=0,
\end{equation}
which is smooth in this chart unless $a=0$.

\item {\bf Chart associated with $v$:}
The equations are
$$y=vy_1, \ t=vt_1, \  s^2+x=vw_1$$  
and
$$y_1w_1+xt_1^2+F(x,vy_1,1)=0,  s^2+x=vw_1.$$
The exceptional divisor is given by $v=0$, i.e. 
$$y_1w_1+xt_1^2+(x-1)^2=0, \ s^2+x=0.$$
The blow up is smooth at any point of the
exceptional divisor except for
$$y_1=v=w_1=t_1=0, \ x=1, s=\pm i.$$

Thus the only singularities are along the proper transform
of $\mr_y$.
The fiber above $s=a$, $x=-a^2$,  $y=v=t=0$, is given by
\begin{equation}\label{f4}
y_1w_1-a^2t_1^2+(1+a^2)^2=0,\\
\end{equation}
which is smooth in this chart unless $a=\pm i$.
\end{itemize}
 
\subsubsection*{Singularities above $\mp_z$.}
Our goal is to show explicitly that the singularity
of the blow up in the exceptional divisor $E_{1,y}$
over $(x,y,z)=(0,0,1)=\mp_z$ is resolved on
blowing up the proper transform of $C_x$.
It suffices to examine the chart associated with $t$,
where we have equation
$$
y_1w_1+x+F(x,ty_1,1)v_1^2=0, \ s^2+x=tw_1,
$$ 
i.e.,
  \begin{equation}\label{resch}
   (y_1+t)w_1-s^2+F(-s^2+tw_1,ty_1,1)v_1^2=0, \ s^2+x=tw_1,
  \end{equation}
  and the proper transform of $C_x$ satisfies
$$y_1+t=0, \ w_1=s=v_1=0.$$
If we compute the singular locus for 
the equation (\ref{resch}) above, at the points of 
the exceptional divisor $t=0$ and above $x=0$, 
we recover the equations for the proper transform
of $C_x$ in this chart.
   
We analyze $X_2$, the blowup along the proper
transform of $C_x$. In any chart above $y_1=t=0$ 
we have $F=1$ so \'etale locally we can introduce
a new variable
$v_2=\sqrt{F} v_1$ to obtain
$$(y_1+t)w_1-s^2+v_2^2=0.$$
After the change of variables $y_2=y_1+t$: 
$$y_2w_1-s^2+v_1^2=0, $$
the singular locus is $y_2=s=w_1=v_2=0.$ 
Here $t$ is a free variable corresponding to an $\mathbb A^1$-factor.
This is the product of an ordinary threefold double point with curve,
thus is resolved on blowing up the singular locus. Note the exceptional
divisor is a smooth quadric surface bundle over the proper transform of $C_x$,
over this chart. (There is a singular fiber over the point where it meets
the proper transform of $C_z$.)

\subsubsection*{Singularities above $\mr_y=(1,0,1)\in \bP^2$.}
    
By the analysis above, we have only to consider the chart of the first blowup associated with $v$.
Recall that it is obtained by setting 
 $$
 y=vy_1, \ t=vt_1, \  s^2+x=vw_1
 $$ 
with equation
 $$
y_1w_1+xt_1^2+F(x,vy_1,1)=0.
 $$ 
The exceptional divisor is given by $v=0$.
The proper transform $R'_y$ of the conic 
$$
 R_y: x-1=y=t=0, \ s^2+1-4v^2=0
$$ 
is then
 \begin{equation}\label{strtrc2}
 x-1=y_1=t_1=0, \ w_1-4v=0, \ s^2+1-4v^2=0.
 \end{equation}
We obtain these equations by inverting the local equation
for the exceptional divisor.
Eliminating $x$ from the equation of $X_1$ yields
an equation that can be put in the form
$$
y_1(w_1-4v)+(-s^2+vw_1)t_1^2+(s^2-vw_1+vy_1+1)^2=0.
$$
Writing $w_2=w_1-4v$ we obtain
$$
y_1w_2+(-s^2+vw_2+4v^2)t_1^2+(s^2-vw_2-4v^2+vy_1+1)^2=0.
$$
The curve $R'_y$ may be expressed as a complete intersection
$$y_1=w_2=t_1=0, \quad \sigma:=(s^2-4v^2+1)+v(y_1-w_2)=0;$$
the coefficient 
$$c:=-s^2+vw_2+4v^2$$
is non-vanishing along $R'_y$ in this chart so we may introduce an
\'etale local coordinate $t_2=\sqrt{c}t_1$.
Then our equation takes the form
$$y_1w_2+t_2^2+\sigma^2=0.$$
In other words, we have ordinary threefold double points along
each point of $R'_y$. Blowing up $R'_y$ resolves the singularity,
and the exceptional divisor over $R'_y$ is fibered in smooth quadric surfaces.

\subsection{$\CH_0$-triviality of the resolution}
Let $E_{1,y}$ denote the exceptional divisor after blowing
up $C_y$. We've seen that
the projection $E_{1,y}\ra C_y$ is a quadric
surface bundle. The fibers are smooth away from $\mq_x,
\mq_z,$ and $\mr_{y\pm}$; over these points the fibers are quadric
cones.

Let $E_{1,x}$ denote the exceptional divisor after
blowing up the proper transform $C'_x$ of $C_x$. 
The fibration $E_{1,x} \ra C'_x$ is also a quadric surface
bundle. The fibers are smooth away from $\mq_y$ and
$\mr_{x\pm}$, where the fibers are quadric cones.

Let $E_{1,z}$ denote the exceptional divisor on
blowing up the proper transform $C'_z$ of $C_z$,
after the first two blow ups. Again
$E_{1,z} \ra C'_z$ is a quadric surface bundle,
smooth away from $\mr_{z\pm}$; the fibers
over these points are quadric cones.

Finally, we blow up the proper transforms 
$R'_x,R'_y,R'_z$ of the disjoint vertical conics.
The local computations above show that the resulting
fourfold $\tilde{X}$ is smooth and the exceptional divisors
$$E_{2,x} \ra R'_x, \ 
E_{2,y} \ra R'_y, \ 
E_{2,z} \ra R'_z,$$
are smooth  rational quadric surface bundles. 

To summarize, fibers of $\beta:\tilde{X} \ra X$ are one of the following:
\begin{itemize}
\item{if $x$ is not contained in any of conics, $\beta^{-1}(x)$ is a point;}
\item{if $x$ is contained in exactly one of the conics, $\beta^{-1}(x)$
is a smooth quadric surface isomorphic to $\bF_0$; when $x$ is a generic point
of one of the conics, then $\beta^{-1}(x)$ is rational over the residue field of $x$;}
\item{if $x$ is contained in two of the conics, $\beta^{-1}(x)=\bF_0 \cup_{\Sigma}\bF_2$,
where $\bF_2$ is the proper transform of a quadric cone appearing as a degenerate
fiber, $\Sigma \subset \bF_2$ is the $(-2)$ curve, and $\Sigma\subset \bF_0$ has
self-intersection $2$.}
\end{itemize}
An application of Proposition~\ref{prop:restrivial} yields that $\beta$
is universally $\CH_0$-trivial.

\section{Analysis of Hodge classes}
\label{sect:rational}

Our approach follows Section 2 of \cite{voisin-stable}. As explained
in Proposition~\ref{prop:ratmain}, a quadric surface bundle over a rational surface $\pi:X\ra S$ is rational
provided $X$ admits an integral class of type $(2,2)$
meeting the fibers of $\pi$ in odd degree.
Here we analyze how these classes occur.

We start by reviewing the
Hodge-theoretic inputs.
Let $\cY\ra B$ be the family of all smooth hypersurfaces in $\bP^2 \times \bP^3$ of bidegree $(2,2)$, 
i.e., $B$ is the complement of the discriminant in
$\bP(\Gamma(\cO_{\bP^2 \times \bP^3}(2,2)))$. For each $b \in B$, let $Y_b$ denote the fiber over $b$.
The Lefschetz hyperplane theorem gives Betti/Hodge numbers
\begin{itemize}
\item
$b_{2i+1}(Y_b)=0$
\item 
$b_2(Y_b)=h^{1,1}(Y_b)=2, \ b_6(Y_b)=h^{3,3}(Y_b)=2$.
\end{itemize}
We compute $b_4(Y_b)$ by analyzing 
$Y_b \ra \bP^2$;
its degeneracy divisor is an octic plane curve $D_b$, of genus $21$.
Indeed, the fibers away from $D_b$ are smooth quadric surfaces
and the fibers over $D_b$ are quadric cones, so we have
$$\begin{array}{rcl}
\chi(Y_b)&=&\chi(\bP^1 \times \bP^1)\chi(\bP^2 \setminus D_b)+
\chi(\text{quadric cone})\chi(D_b) \\
	 &=& 4\cdot (3-(-40))+3\cdot(-40)=52.
\end{array}
$$
It follows that $b_4(Y_b)=46$.

We extract the remaining Hodge numbers using techniques of Griffiths
and Donagi \cite{DG} for hypersurfaces in projective space, extended
to the toric case by Batyrev and Cox.
Let $F$ be the defining equation of bidegree $(2,2)$ and consider the bigraded 
{\em Jacobian ring}:
$$
\Jac(F)= \bC[x,y,z; s,t,u,v] / \cI(F),
$$
where $\cI(F)$ is the ideal of partials of $F$. 
Note the partials satisfy Euler relations:
\begin{equation} \label{euler}
x\frac{\partial F}{\partial x}+
y\frac{\partial F}{\partial y}+
z\frac{\partial F}{\partial z}= 2F = 
s\frac{\partial F}{\partial s}+
t\frac{\partial F}{\partial t}+
u\frac{\partial F}{\partial u}+
v\frac{\partial F}{\partial v}.
\end{equation}
Consider the {\em vanishing cohomology} 
$$H^4(Y_b)_{van}:=H^4(Y_b)/H^4(\bP^2 \times \bP^3),$$
i.e., we quotient by $\left<h_1^2,h_1h_2,h_2^2\right>$
where $h_1$ and $h_2$ are pull-backs of the hyperplane classes
of $\bP^2$ and $\bP^3$ respectively.
Then we have \cite[Theorem 10.13]{batyrev-cox}:
\begin{itemize}
\item 
$H^{4,0}(Y_b)=H^{4,0}(Y_b)_{van} = \Jac(F)_{(-1,-2)} = 0$
\item 
$H^{3,1}(Y_b) = H^{3,1}(Y_b)_{van} \simeq 
\Jac(F)_{(1,0)} = \bC[x,y,z]_1 \simeq \bC^3$
\item
$H^{2,2}(Y_b)_{van} \simeq  \Jac(F)_{(3,2)}\simeq \bC^{37}$
\item 
$H^{1,3}(Y_b)=H_{1,3}(Y_b)_{van} \simeq  \Jac(F)_{(5,4)}\simeq \bC^3.$
\end{itemize}
The first two dimension computations imply the others by the formula
$$b_4(Y_b)=\sum_{p+q=4} h^{p,q}(Y_b);$$
or one may compute the Hilbert function of an ideal
generated by three forms of degree $(1,2)$ and four forms
of degree $(2,1)$, subject to the relations (\ref{euler}) but
otherwise generic.

We recall the technique of \cite[5.3.4]{voisin-book}:
\begin{prop}
Suppose there exists a $b_0\in B$ and $\gamma \in  H^{2,2}(Y_{b_0})$ 
such that the infinitesimal period map evaluated at $\gamma$
$$
\bar{\nabla}(\gamma):T_{B,b_0} \ra H^{1,3}(Y_{b_0})
$$
is surjective. Then for any $b\in B$ and any Euclidean neighborhood
$b\in B' \subset B$, the image of the natural map (composition of
inclusion with local trivialization):
$$
T_b: H^{2,2}(Y_{B'},\bR) \ra H^4(Y_b,\bR)
$$
contains an open subset $V_b \subset H^4(Y_b,\bR)$.
\end{prop}
Note that the image is the set of real degree-four classes that are of
type $(2,2)$ for some $b' \in B'$.
We should point out that our variation of Hodge structures is 
of weight two after a suitable Tate twist.

The infinitesimal condition is easy to check here.
Since 
$$B\subset \bP(\Gamma(\cO_{\bP^2 \times \bP^3}(2,2)))$$
we may identify
$$T_{B,b_0}=
(\bC[x,y,z; s,t,u,v] / \left<F_0\right>)_{(2,2)},$$
where $F_0$ is the defining equation of $Y_{b_0}$.
The infinitesimal period map
$$
T_{B,b_0} \ra \Hom(H^{2,2}(Y_b),H^{1,3}(Y_b))$$
is given by multiplication
$$(\bC[x,y,z; s,t,u,v]/\left<F_0\right>)_{(2,2)} \times
\Jac(F_0)_{(3,2)} \ra \Jac(F_0)_{(5,4)}.$$
For fixed $\gamma \in \Jac(F_0)_{(3,2)}$, 
the differential in Voisin's hypothesis is computed
by multiplying $\gamma$ with the elements of bidegree $(2,2)$ 
\cite[Theorem~6.13]{voisin-book}.

\begin{exam}
Consider the hypersurface $Y_{b_0} \subset \bP^2 \times \bP^3$
with equation
$$ \begin{array}{rcl}
F_0&=& (u^2+uv+ts) x^2+
(-t^2+u^2-v^2-s^2)xy +
(t^2+uv+ts) y^2\\ & & +
(-t^2+u^2-v^2-s^2)xz + 
(t^2-16tu-u^2+v^2+s^2)yz\\
   & & +
(-3uv-3ts+s^2)z^2.
\end{array}
$$
We computed the Jacobian ring using Macaulay 2 \cite{M2}. In particular,
we verified that
\begin{itemize} 
\item{$\Jac(F_0)_{(m_1,m_2)}=0$ for 
$$(m_1,m_2) \ge (13,2), (7,3), (3,5),$$
so in particular $Y_{b_0}$
is smooth;}
\item{the monomials $\{xz^4v^4, yz^4v^4, z^5v^4\}$ form a basis
for $\Jac(F_0)_{(5,4)}$.}
\end{itemize}
Setting $\gamma=z^3v^2$, the multiples of $\gamma$ 
generate $\Jac(F_0)_{(5,4)}$.
Hence this example satisfies Voisin's hypothesis on the differential
of the period map.
\end{exam}

\begin{prop} \label{prop:NL}
Consider the Noether-Lefschetz loci 
$$
\begin{array}{r} \{b \in B : Y_b \text{ admits an integral class of type $(2,2)$ meeting}\\
    			      \text{the fibers of $Y_b \ra \bP^2$ in odd degree} \}.
\end{array}
$$
These are dense in the Euclidean topology on $B$.
\end{prop}
\begin{proof}
Let $\Lambda$ be the middle cohomology lattice of $Y_b$. 
It remains to check that classes in $\Lambda$ with odd degree with respect to $Y_b \ra \bP^2$
are dense in $\Lambda \otimes \bQ$. 
There exist odd-degree classes as we can write down a bidegree $(2,2)$ hypersurface 
containing a constant section of $\bP^2 \times \bP^3 \ra \bP^2$. 
Adding even degree classes to this fixed class yields
a dense collection of the desired classes.
\end{proof}

The Noether-Lefschetz loci produced by this argument have
codimension at most three in moduli; each is an algebraic subvariety
of
$B \subset \bP(\Gamma(\cO_{\bP^2\times\bP^3}(2,2)))\simeq \bP^{59}$ \cite{CDK}.
Any projective threefold in $\bP^{59}$
will meet the closures of 
infinitely many of these loci.

\section{Proof of Theorem~\ref{theo:main}}
We assemble the various ingredients developed above:
\begin{enumerate}
\item{Theorem~\ref{thm:main-help} guarantees that a very general
hypersurface of bidegree $(2,2)$ in $\bP^2 \times \bP^3$ fails to
be stably rational, provided we can find a special $X$ satisfying
its hypotheses.}
\item{The candidate example is introduced in Example~\ref{exam:basic}.}
\item{In Section~\ref{sect:brauer}, we show that $X$ has non-trivial
unramified second cohomology. This verifies the first hypothesis
of Theorem~\ref{thm:main-help}.}
\item{In Section~\ref{sect:sing}, we analyze the singularities of $X$,
checking that it admits a resolution with universally $\CH_0$-trivial
fibers.}
\item{Proposition~\ref{prop:ratmain} gives a cohomological
sufficient condition for rationality of $(2,2)$ hypersurfaces
in $\bP^2 \times \bP^3$; Proposition~\ref{prop:NL} shows this
condition is satisfied over a dense subset of the moduli space.}
\end{enumerate}
Consider a family $\phi: \cX \ra B$ of smooth 
$(2,2)$ hypersurfaces in $\bP^2\times \bP^3$
over a connected base $B$. If the base meets both the locus parametrizing
non-stably rational varieties and the Noether-Lefschetz loci then $\phi$
has both rational and irrational fibers.

\bibliographystyle{alpha}
\bibliography{quadrics}
\end{document}